\documentclass{amsart}
\usepackage{amsfonts,amssymb,amsmath,amsthm}
\usepackage{url}
\usepackage{enumerate}

\newtheorem{thrm}{Theorem}[section]
\newtheorem{lem}[thrm]{Lemma}

\newtheorem{cor}[thrm]{Corollary}
\theoremstyle{definition}

\numberwithin{equation}{section}

\newcommand{\R}{\mathbb{R}}

\newcommand{\X}{\mathcal{X}}
\newcommand{\M}{M}
\newcommand{\XM}{\X(\M)}

\newcommand{\D}{\operatorname{D}}

\newcommand{\rank}{\operatorname{rank}}
\newcommand{\sgn}{\operatorname{sgn}}

\author[M. Szancer]{Michal Szancer}
\author[Z. Szancer]{Zuzanna Szancer}
\address{
Michal Szancer \\
Gornikow Street 21/1\\
30-819 Krakow, Poland}
\email{mszancer@gmail.com}

\address{Zuzanna Szancer \\ Department of Applied Mathematics \\ University of Agriculture in Krakow\\ 253c Balicka Street\\
30-198 Krakow, Poland}
\email{Zuzanna.Szancer@urk.edu.pl}

\begin{document}

\title[On $4$-dimensional Lorentzian affine hypersurfaces...]{On $4$-dimensional Lorentzian affine hypersurfaces with an almost symplectic form}

\begin{abstract}
In this paper we study $4$-dimensional affine hypersurfaces with a Lorentzian second fundamental form additionally equipped with an almost symplectic structure $\omega$. We prove that the rank of the shape operator is at most one if $R^k\cdot \omega=0$ or $\nabla^k\omega=0$
for some positive integer $k$. This result is the final step in a classification of Lorentzian affine hypersurfaces with higher order parallel almost symplectic forms.
\end{abstract}
\maketitle

\section{Introduction} \label{sect1}
\label{intro}
\par Parallel structures are of great interest in classical Riemmanian geometry (see \cite{D, Lumiste, BDI})
as well as in affine differential geometry (\cite{BNS, DV2, DVY, HLLVran, HLVran, Hildebrand2015, HLLV2}).
Higher order parallel structures are natural generalization of parallel structures and are widely studied as well
(\cite{D, D2, Vrancken, LucVrancken2, VanderVeken}). There exist also some classification results
in context of induced almost contact and almost paracontact structures (\cite{Szancer2, Szancer5}).

\par On the other hand O. Baues and V. Cort\'{e}s  studied affine hypersurfaces equipped with an almost complex structure (\cite{BC}).
They proved that every simply connected special K\"{a}hler manifold (\cite{Freed}) can be realized in a canonical way as an improper affine hypersphere.
In 2006  V. Cort\'{e}s together with M.-A. Lawn and L. Sch\"{a}fer (\cite{CLS}) proved a similar result for special para-K\"{a}hler manifolds (\cite{CMMS}).
Such hyperspheres were called by the authors special affine hyperspheres. In both cases an important role was played by the K\"{a}hlerian (resp. para-K\"{a}hlerian) symplectic form $\omega$. Later special affine hyperspheres were generalized by the first author in \cite{MSz3}.
These results show that there are interesting relations between symplectic (in particular K\"{a}hler and para-K\"{a}hler) geometry
and affine differential geometry.


\par Motivated by the above results as well as M. Kon results (\cite{Kon}) the first author studied affine hypersurfaces
$f\colon M\rightarrow \R^{2n+1}$ with a transversal vector field $\xi$ additionally equipped with an almost symplectic structure $\omega$.
In \cite{MSz} the following result was obtained:
\begin{thrm}[\cite{MSz}]\label{tw::RXYOmega}
Let $f\colon \M\rightarrow\R^{2n+1}$ be a non-degenerate affine hypersurface with a transversal vector field  $\xi$ and an almost symplectic form $\omega$.
Equality $R(X,Y)\omega=0$ for every $X,Y\in\XM$ holds if and only if  $\dim M=2$ and $\xi$ is locally equiaffine or $\dim M\geq 4$ and $\nabla$ is flat.
\end{thrm}

In the case when the second fundamental form is positive definite and the transversal vector field $\xi$ is locally equiaffine the above theorem
generalizes to an arbitrary power of $R$. Namely, we have
\begin{thrm}[\cite{MSz}]\label{tw::RKOmega0}
Let $f\colon \M\rightarrow\R^{2n+1}$ be a non-degenerate affine hypersurface ($\dim M\geq 4$) with a locally equiaffine transversal vector field  $\xi$
and an almost symplectic form $\omega$. Additionally assume that the second fundamental form is positive definite on $M$.
If $R^l\omega=0$ for some positive integer $l$ then $\nabla$ is flat.
\end{thrm}
As a consequence of the above theorem we obtain
\begin{thrm}[\cite{MSz}]\label{tw::NablaKOmega0}
Let $f\colon \M\rightarrow\R^{2n+1}$ be a non-degenerate affine hypersurface ($\dim M\geq 4$) with a locally equiaffine transversal vector field  $\xi$
and an almost symplectic form $\omega$. Additionally assume that the second fundamental form is positive definite on $M$.
If $\nabla^k\omega=0$ for some positive integer $k$ then $\nabla$ is flat.
\end{thrm}
Later in \cite{MSz2} it was shown that although the above theorems are not true in general when the second fundamental form is Lorentzian, we still have
strong constrains on the shape operator if only $\dim M\geq 6$. Namely we have the following theorems:
\begin{thrm}[\cite{MSz2}]\label{tw::RKomega0Lorentzian}
Let $f\colon \M\rightarrow\R^{2n+1}$ \mbox{($\dim M\geq 6$)} be a non-degenerate affine hypersurface with a locally equiaffine transversal vector field  $\xi$
and an almost symplectic form $\omega$. If $R^k\omega=0$ for some $k\geq 1$ and the second fundamental form is Lorentzian on $M$ (that is has signature $(2n-1,1)$)
then the shape operator $S$ has the rank $\leq 1$.
\end{thrm}

\begin{thrm}[\cite{MSz2}]\label{tw::NablaKOmega0Lorentzian}
Let $f\colon \M\rightarrow\R^{2n+1}$ ($\dim M\geq 6$) be a non-degenerate affine hypersurface with a locally equiaffine transversal vector field  $\xi$
and an almost symplectic form $\omega$. If $\nabla^k\omega=0$ for some $k\geq 1$ and the second fundamental form is Lorentzian on $M$ (that is has signature $(2n-1,1)$)
then the shape operator $S$ has the rank $\leq 1$.
\end{thrm}

\par The main purpose of this paper is to prove that Theorem \ref{tw::RKomega0Lorentzian} and Theorem \ref{tw::NablaKOmega0Lorentzian} hold also for
$4$-dimensional affine hypersurfaces. Although some results obtained in \cite{MSz2} stay true in $4$-dimensional case, the key step of proof cannot be
easily repeated. Simply there is not enough "room" in $4$-dimensional space and results from \cite{MSz2} do not provide enough information about
structure of eigen values of the shape operator. For this reason in this paper we need to develop a bit different methods. In particular, we consider
two separate cases and find several new properties of $R^k\omega$ tensor.


\par In Section 2 we briefly recall the basic formulas of affine differential
geometry. We also recall some basic definitions from symplectic geometry that will be used later in this paper.

\par The Section 3 contains the main results of this paper. We show that if there exists an almost
symplectic structure $\omega$ satisfying condition $R^k\cdot \omega=0$ or $\nabla^k\omega=0$ for some positive integer $k$ then the shape operator must have a very special form.
More precisely, we obtain that the rank of the shape operator $S$ must
be $\leq 1$ if only the transversal vector field is locally equiaffine.

\section{Preliminaries}
\label{sec:1}
We briefly recall the basic formulas of affine differential
geometry. For more details, we refer to \cite{NomSas}. Let $f\colon M\rightarrow\R^{n+1}$ be an orientable
connected differentiable $n$-dimensional hypersurface immersed in
the affine space $\R^{n+1}$ equipped with its usual flat connection
$\D$. Then for any transversal vector field $\xi$ we have
\begin{equation}\label{eq::FormulaGaussa}
\D_Xf_\ast Y=f_\ast(\nabla_XY)+h(X,Y)\xi
\end{equation}
and
\begin{equation}\label{eq::FormulaWeingartena}
\D_X\xi=-f_\ast(SX)+\tau(X)\xi,
\end{equation}
where $X,Y$ are vector fields tangent to $M$. It is known that $\nabla$ is a torsion-free connection, $h$ is a symmetric
bilinear form on $M$, called \emph{the second
fundamental form}, $S$ is a tensor of type $(1,1)$, called \emph{the
shape operator}, and $\tau$ is a 1-form, called \emph{the transversal connection form}.
The vector field $\xi$ is called \emph{equiaffine} if $\tau=0$.  When $d\tau=0$ the vector field $\xi$ is called \emph{locally equiaffine}.

\par When $h$ is non-degenerate then $h$ defines a
pseudo-Rie\-man\-nian metric on $M$. In this case we say that the hypersurface or the
hypersurface immersion is \emph{non-degenerate}. In this paper we always assume that $f$ is
non-degenerate.  We have the following
\begin{thrm}[\cite{NomSas}, Fundamental equations]\label{tw::FundamentalEquations}
For an arbitrary transversal vector field $\xi$ the induced
connection $\nabla$, the second fundamental form $h$, the shape
operator $S$, and the 1-form $\tau$ satisfy
the following equations:
\begin{align}
\label{eq::Gauss}&R(X,Y)Z=h(Y,Z)SX-h(X,Z)SY,\\
\label{eq::Codazzih}&(\nabla_X h)(Y,Z)+\tau(X)h(Y,Z)=(\nabla_Y h)(X,Z)+\tau(Y)h(X,Z),\\
\label{eq::CodazziS}&(\nabla_X S)(Y)-\tau(X)SY=(\nabla_Y S)(X)-\tau(Y)SX,\\
\label{eq::Ricci}&h(X,SY)-h(SX,Y)=2d\tau(X,Y).
\end{align}
\end{thrm}
The equations (\ref{eq::Gauss}), (\ref{eq::Codazzih}),
(\ref{eq::CodazziS}), and (\ref{eq::Ricci}) are called the
equations of Gauss, Codazzi for $h$, Codazzi for $S$ and Ricci,
respectively.




\par Let $\omega$ be a non-degenerate 2-form on manifold $M$. The form $\omega$
we call an \emph{almost symplectic structure}. It is easy to see that if a manifold $M$ admits some almost symplectic structure then M is orientable manifold of even dimension.
Structure $\omega$ is called a \emph{symplectic structure}, if it is almost symplectic and additionally satisfies $d\omega=0$.
Pair $(M,\omega)$ we call \emph{(almost) symplectic manifold}, if $\omega$ is (almost) symplectic structure on $M$.

\par Recall (\cite{AlbPic}) that affine connection $\nabla$ on an almost symplectic manifold $(M,\omega)$ we call an \emph{almost symplectic connection} if $\nabla\omega=0$.
An affine connection $\nabla$ on an almost symplectic manifold $(M,\omega)$ we call a \emph{symplectic connection} if it is almost symplectic and torsion-free.

\par For a tensor field $T$ of type $(0,p)$ its covariant derivation $\nabla T$ is a tensor field of type $(0,p+1)$ given by the formula:
\begin{align*}
(\nabla T)(X_1,X_2,\ldots,X_{p+1}):=X_1(T(X_2,\ldots,X_{p+1}))\\ -\sum_{i=2}^{p+1}T(X_2,\ldots,\nabla_{X_1}X_i,\ldots,X_{p+1}).
\end{align*}
Higher order covariant derivatives of $T$ can be defined by recursion:
\begin{align*}
(\nabla^{k+1} T)=\nabla(\nabla^kT).
\end{align*}
To simplify computation it is often convenient to define $\nabla^0T:= T$.
\par If $R$ is a curvature tensor for an affine connection  $\nabla$, one can define a new tensor  $R\cdot T$ of type $(0,p+2)$ by the formula
\begin{align*}
(R\cdot T)(X_1,X_2,\ldots,X_{p+2}):= -\sum_{i=3}^{p+2}T(X_3,\ldots,R(X_1,X_2)X_i,\ldots,X_{p+2}).
\end{align*}
Analogously to the previous case, we may define a tensor $R^k\cdot T$ of type $(0,2k+p)$ using the following recursive formula:
$$
R^k\cdot T=R\cdot (R^{k-1}\cdot T)
$$
and additionally $R^0\cdot T:=T$.

\section{Hypersurfaces with "higher order" parallel symplectic structure}
\label{sec:3}
In this section we study properties of $4$-dimensional affine hypersurfaces $f\colon \M\rightarrow\R^{5}$ with a Lorentzian second fundamental form.
We assume that our hypersurfaces are equipped with an almost symplectic structure $\omega$ satisfying condition $R^k\omega=0$ for some positive integer $k$.
In particular we obtain constrains on hypersurfaces with the property $\nabla^k\omega=0$.


\par First we recall the following lemma from \cite{MSz}.

\begin{lem}[\cite{MSz}]\label{lm::WzorNaRkT}
Let $T$ be a tensor of  type $(0,p)$ and let $\nabla$ be an affine torsion-free connection. Then for every $k\geq 1$ and for any $2k+p$ vector fields
$X_{\pm 1}^1,\ldots,X_{\pm 1}^k$, $Y_1,\ldots,Y_p$ the following identity holds:
\begin{align}\label{eq::WzorNaRkT}
(R^k&\cdot T)(X_1^1,X_{-1}^1,\ldots,X_1^k,X_{-1}^k,Y_1,\ldots,Y_p)\\
\nonumber &=\sum_{a\in\mathcal{J}}\sgn a(\nabla^{2k}T)(X_{a(1)}^1,X_{-a(1)}^1,\ldots,X_{a(k)}^k,X_{-a(k)}^k,Y_1,\ldots,Y_p),
\end{align}
where $\mathcal{J}=\{a\colon I_k\rightarrow \{-1,1\}\}$ and $\sgn a:=a(1)\cdot\ldots\cdot a(k)$.
\end{lem}
In order to simplify the notation, we will be often omitting "$\cdot$" in $R^k\cdot T$ when no confusion arises.
Thus we will be writing often $R^k T$ instead of $R^k\cdot T$.

In all the below lemmas we assume that $f\colon \M\rightarrow\R^{5}$ is a non-degenerate affine hypersurface with a locally equiaffine transversal vector field  $\xi$ and an almost symplectic form $\omega$.
About objects $\nabla$, $h$, $S$ and $\tau$ we assume that they are induced by $\xi$.

First note that combining Lemma 3.6 and Lemma 3.11 from \cite{MSz2} and adapting it to 4-dimensional case we have the following:


\begin{lem}[\cite{MSz2}]\label{lm::PrzypNiemozliwy}
Let $f\colon \M\rightarrow\R^{5}$ be a non-degenerate Lorentzian affine hypersurface with a locally equiaffine transversal vector field  $\xi$
and an almost symplectic form $\omega$.
If $R^k\omega=0$ for some $k\geq 1$ then for every point $x\in M$ there exists a basis ${e_1,\ldots,e_{4}}$ of $T_xM$
such that the shape operator $S$ and the second fundamental form  $h$ can be expressed in this basis either in the form
\begin{equation}\label{eq::Sh21}
S=\left[\begin{matrix}
\lambda_1 & 0 &  0 & 0\\
0 & \lambda_2 &  0 & 0 \\
0 & 0 &  \lambda_{3} & 0 \\
0 & 0 &  0 & \lambda_{4}
\end{matrix}\right]
h=\left[\begin{matrix}
1 & 0  & 0 & 0\\
0 & 1  & 0 & 0 \\
0 & 0  & 1 & 0 \\
0 & 0 & 0 & -1
\end{matrix}\right],
\end{equation}
where $\lambda_1,\ldots,\lambda_{4}\in\R$, or in the form
\begin{equation}\label{eq::Sh22}
S=\left[\begin{matrix}
\lambda_1 & 0 & 0 & 0\\
0 & \lambda_2 & 0 & 0 \\
0 & 0 &  \alpha & \gamma \\
0 & 0 &  -\gamma & \beta
\end{matrix}\right]
h=\left[\begin{matrix}
1 & 0 & 0 & 0\\
0 & 1 & 0 & 0 \\
0 & 0  & 1 & 0 \\
0 & 0  & 0 & -1
\end{matrix}\right],
\end{equation}
where $\lambda_1,\lambda_{2},\alpha,\beta,\gamma\in\R, \gamma\neq 0$.
\end{lem}

Let us recall yet another lemma from \cite{MSz2} (again adapted to 4-dimensional case).


\begin{lem}[\cite{MSz2}]\label{lm::LematZDet}
If $S$ and $h$ are of the form {\upshape(\ref{eq::Sh22})} then for every $k\geq 1$ we have
\begin{eqnarray}\label{eq::WzorR2kA}
R^{2k}\omega(\underbrace{e_{3},e_{4},\ldots,e_{3},e_{4}}_{4k},e_i,e_{4})\\ \nonumber =\det\left[
\begin{matrix}
\alpha & \gamma \\
-\gamma & \beta
\end{matrix}
\right]^k \omega(e_{i},e_{4})
\end{eqnarray}
if $i < 3$,
\begin{eqnarray}\label{eq::WzorR2k+1B}
R^{2k+1}\omega(\underbrace{e_{3},e_{4},\ldots,e_{3},e_{4}}_{4k},e_1,X,e_1,X)\\ \nonumber =4^k\gamma\det\left[
\begin{matrix}
\alpha & \gamma \\
-\gamma & \beta
\end{matrix}
\right]^k \omega(e_{3},e_{4})
\end{eqnarray}
for $X=e_{3}$ or $X=e_{4}$,
\begin{eqnarray}\label{eq::WzorR2k+1C}
R^{2k+1}\omega(\underbrace{e_{3},e_{4},\ldots,e_{3},e_{4}}_{4k},e_1,X,e_1,Y)\\ \nonumber =2\cdot 4^{k-1}(\alpha-\beta)\det\left[
\begin{matrix}
\alpha & \gamma \\
-\gamma & \beta
\end{matrix}
\right]^k \omega(e_{3},e_{4}),
\end{eqnarray}
for $X=e_{3}$ and $Y=e_{4}$ or $X=e_{4}$ and $Y=e_{3}$.
\end{lem}

Thanks to the above lemma we have the following:
\begin{cor}\label{cor::detZero}
If $S$ and $h$ are of the form {\upshape(\ref{eq::Sh22})} and $R^k\omega=0$ for some $k\geq 1$ then
$$
\det\left[
\begin{matrix}
\alpha & \gamma \\
-\gamma & \beta
\end{matrix}
\right]=\alpha\beta+\gamma^2=0.
$$
\end{cor}
\begin{proof}
If  $R^k\omega=0$ then  $R^{2k}\omega=0$ and $R^{2k+1}\omega=0$.
Since $\omega$ is non-degenerate we can find $i<4$ such that $\omega(e_i,e_{4})\neq 0$. If $i=1$ or $i=2$ then by formula (\ref{eq::WzorR2kA}) we get
$\alpha\beta+\gamma^2=0$.
If $i=3$ then by formula (\ref{eq::WzorR2k+1B}) we again obtain  $\alpha\beta+\gamma^2=0$ (since $\gamma\neq 0$).
\end{proof}

Now, we shall consider two separate cases: when $\beta^2-\gamma^2\neq 0$ and when $\beta^2-\gamma^2=0$. In the first case, using suitable change of the basis one may show that $S$ is diagonalisable. Namely, we have


\begin{lem}\label{lm::BetaGamma}
If $S$ and $h$ are of the form {\upshape(\ref{eq::Sh22})} and $\beta^2-\gamma^2\neq 0$ and $R^k\omega=0$ for some $k\geq 1$ then there exists
a basis ${e_1',\ldots,e_{4}'}$ of $T_xM$
such that the shape operator $S$ and the second fundamental form  $h$ can be expressed in this new basis in the following form:
\begin{equation}\label{eq::Sh22new}
S=\left[\begin{matrix}
\lambda_1 & 0 & 0 & 0\\
0 & \lambda_2 & 0 & 0 \\
0 & 0 & \alpha+\beta & 0 \\
0 & 0 &  0 & 0
\end{matrix}\right]
h=\left[\begin{matrix}
1 & 0 & 0 & 0\\
0 & 1 & 0 & 0 \\
0 & 0  & -\epsilon & 0 \\
0 & 0  & 0 & \epsilon
\end{matrix}\right]
\end{equation}
where $\epsilon = \sgn (\beta^2-\gamma^2)\in\{1,-1\}$.
\end{lem}
\begin{proof}
Let as define a matrix
$$
P=\left[\begin{matrix}
1 & 0 & 0 & 0\\
0 & 1 & 0 & 0 \\
0 & 0 &  \frac{\gamma}{\sqrt{|\beta^2-\gamma^2|}} & \frac{\beta}{\sqrt{|\beta^2-\gamma^2|}} \\
0 & 0 &  \frac{\beta}{\sqrt{|\beta^2-\gamma^2|}} & \frac{\gamma}{\sqrt{|\beta^2-\gamma^2|}}
\end{matrix}\right].
$$
Since $\det P = \pm 1$ the matrix $P$ is non-singular and we can define a new basis of $T_xM$ by the formula $e_i':=Pe_i$ for $i=1,\ldots,4$.
By straightforward computations we check that $S$ and $h$ in this new basis take the form:
\begin{equation}
S=\left[\begin{matrix}
\lambda_1 & 0 & 0 & 0\\
0 & \lambda_2 & 0 & 0 \\
0 & 0 &  \frac{\beta^3-2\beta\gamma^2-\alpha\gamma^2}{\beta^2-\gamma^2} & -\frac{(\alpha\beta+\gamma^2)\gamma}{\beta^2-\gamma^2} \\
0 & 0 &  \frac{(\alpha\beta+\gamma^2)\gamma}{\beta^2-\gamma^2} & -\frac{(\alpha\beta+\gamma^2)\beta}{\beta^2-\gamma^2}
\end{matrix}\right]
h=\left[\begin{matrix}
1 & 0 & 0 & 0\\
0 & 1 & 0 & 0 \\
0 & 0  & -\epsilon & 0 \\
0 & 0  & 0 & \epsilon
\end{matrix}\right]
\end{equation}
Eventually, using Corollary \ref{cor::detZero} we see that $S$ simplify to (\ref{eq::Sh22new}).
\end{proof}
When $\beta^2-\gamma^2=0$ the situation is much more complicated. In this case we have $\alpha=\pm\gamma$ and $\beta=\mp\gamma$. Most part of this section is devoted to this case.

In order to simplify further computations, let us introduce the following notation:
\begin{align*}
A_k:=R^k\omega(e_1,e_4,e_3,e_4,\overbrace{e_3,e_4,\ldots,e_3,e_4}^{2k-2});\\
B_k:=R^k\omega(e_3,e_4,e_1,e_4,\overbrace{e_3,e_4,\ldots,e_3,e_4}^{2k-2});\\
C_k:=R^k\omega(e_1,e_3,e_3,e_4,\overbrace{e_3,e_4,\ldots,e_3,e_4}^{2k-2});\\
D_k:=R^k\omega(e_3,e_4,e_1,e_3,\overbrace{e_3,e_4,\ldots,e_3,e_4}^{2k-2})
\end{align*}
for $k\geq 1$.


\begin{lem}\label{lm::ABCD}
If $S$ and $h$ are of the form {\upshape(\ref{eq::Sh22})} and $\alpha=\pm\gamma$ and $\beta=\mp\gamma$ then for every $k\geq 1$ we have
\begin{align}
\label{eq::ABCD::1} B_{k+1}=\pm\gamma C_{k}-\gamma A_k,\\
\label{eq::ABCD::2} D_{k+1}=\gamma C_{k}\mp\gamma A_k.
\end{align}
\end{lem}
\begin{proof}
We shall prove only (\ref{eq::ABCD::1}). The proof of (\ref{eq::ABCD::2}) goes in a similar way.
First note, that by the Gauss equation we have
\begin{align}
\label{eq::ABCD::Gauss1}R(e_3,e_4)e_1=R(e_3,e_4)e_2=0, \\
\label{eq::ABCD::Gauss2}R(e_3,e_4)e_3=-S e_4=-\gamma e_3\pm\gamma e_4, \\
\label{eq::ABCD::Gauss3}R(e_3,e_4)e_4 = -S e_3 = \mp\gamma e_3+\gamma e_4.
\end{align}
Now we compute
\begin{align*}
B_{k+1}&=R^{k+1}\omega(e_3,e_4,e_1,e_4,\overbrace{e_3,e_4,\ldots,e_3,e_4}^{2k})\\
       &=R(e_3,e_4)\cdot R^k\omega(e_1,e_4,\overbrace{e_3,e_4,\ldots,e_3,e_4}^{2k})\\
       &=-R^k\omega(R(e_3,e_4)e_1,e_4,\overbrace{e_3,e_4,\ldots,e_3,e_4}^{2k})\\
       &\phantom{=}-R^k\omega(e_1,R(e_3,e_4)e_4,\overbrace{e_3,e_4,\ldots,e_3,e_4}^{2k})\\
       &\phantom{=}-R^k\omega(e_1,e_4,\overbrace{R(e_3,e_4)e_3,e_4,\ldots,e_3,e_4}^{2k})\\
       &\phantom{=}\cdots\\
       &\phantom{=}-R^k\omega(e_1,e_4,\overbrace{e_3,e_4,\ldots,e_3,R(e_3,e_4)e_4}^{2k}).
\end{align*}
Using (\ref{eq::ABCD::Gauss1})--(\ref{eq::ABCD::Gauss3}) we obtain
\begin{align*}
B_{k+1}&=0-R^k\omega(e_1,\mp\gamma e_3+\gamma e_4,\overbrace{e_3,e_4,\ldots,e_3,e_4}^{2k})\\
       &\phantom{=}-R^k\omega(e_1,e_4,\overbrace{-\gamma e_3\pm\gamma e_4,e_4,\ldots,e_3,e_4}^{2k})\\
       &\phantom{=}\cdots\\
       &\phantom{=}-R^k\omega(e_1,e_4,\overbrace{e_3,e_4,\ldots,e_3,\mp\gamma e_3+\gamma e_4}^{2k})\\
       &=\pm\gamma C_k-\gamma A_k\\
       &\phantom{=}+(\gamma A_k-\gamma A_k)+\cdots+(\gamma A_k-\gamma A_k)\\
       &=\pm\gamma C_k-\gamma A_k.
\end{align*}
\end{proof}
Now, let us define a family of 2-forms on $T_xM$ as follows:
\begin{align}
E_k^i(X,Y):=R^k\omega(\overbrace{e_3,e_4}^1,\overbrace{e_3,e_4}^2,\ldots,e_3,e_4,\overbrace{X,Y}^i,e_3,e_4,\ldots,e_3,e_4)
\end{align}
for $k\geq 1$, $i\in\{1,\ldots,k+1\}$.
We have the following lemma:


\begin{lem}\label{lm::Eki}
If $S$ and $h$ are of the form {\upshape(\ref{eq::Sh22})} and $\alpha=\pm\gamma$ and $\beta=\mp\gamma$ then for every $k\geq 2$, $i\in\{3,\ldots,k+1\}$
and $X,Y\in\{e_1,e_3,e_4\}$ we have $E_k^i(X,Y)=0$.
\end{lem}
\begin{proof}
For $k=2$ and $i=3$ by straightforward computation we check that $E_2^3(X,Y)=0$. Assume now that $E_k^i(X,Y)=0$ for some $k\geq 2$ and for every $i\in\{3,\ldots,k+1\}$.
Let us fix $i\in\{3,\ldots,k+2\}$. Then we have
\begin{align*}
E_{k+1}^i(X,Y)&=R^{k+1}\omega(e_3,e_4,e_3,e_4,\ldots,e_3,e_4,\overbrace{X,Y}^{i},e_3,e_4,\ldots,e_3,e_4)\\
              &= -R^k\omega(R(e_3,e_4)e_3,e_4,\ldots,e_3,e_4,\overbrace{X,Y}^{i-1},e_3,e_4,\ldots,e_3,e_4)\\
              &\phantom{=} -R^k\omega(e_3,R(e_3,e_4)e_4,\ldots,e_3,e_4,\overbrace{X,Y}^{i-1},e_3,e_4,\ldots,e_3,e_4)\\
              &\phantom{=} \cdots\\
              &\phantom{=} -R^k\omega(e_3,e_4,\ldots,e_3,e_4,\overbrace{R(e_3,e_4)X,Y}^{i-1},e_3,e_4,\ldots,e_3,e_4)\\
              &\phantom{=} -R^k\omega(e_3,e_4,\ldots,e_3,e_4,\overbrace{X,R(e_3,e_4)Y}^{i-1},e_3,e_4,\ldots,e_3,e_4)\\
              &\phantom{=} \cdots\\
              &\phantom{=} -R^k\omega(e_3,e_4,\ldots,e_3,e_4,\overbrace{X,Y}^{i-1},e_3,e_4,\ldots,R(e_3,e_4)e_3,e_4)\\
              &\phantom{=} -R^k\omega(e_3,e_4,\ldots,e_3,e_4,\overbrace{X,Y}^{i-1},e_3,e_4,\ldots,e_3,R(e_3,e_4)e_4)\\
              &= -R^k\omega(e_3,e_4,\ldots,e_3,e_4,\overbrace{R(e_3,e_4)X,Y}^{i-1},e_3,e_4,\ldots,e_3,e_4)\\
              &\phantom{=} -R^k\omega(e_3,e_4,\ldots,e_3,e_4,\overbrace{X,R(e_3,e_4)Y}^{i-1},e_3,e_4,\ldots,e_3,e_4)
\end{align*}
where the last equality follows from (\ref{eq::ABCD::Gauss2})--(\ref{eq::ABCD::Gauss3}).
The above formula can be rewritten as follows:
$$
E_{k+1}^i(X,Y) = -E_{k}^{i-1}(R(e_3,e_4)X,Y)-E_{k}^{i-1}(X,R(e_3,e_4)Y).
$$
\par Let $i>3$. Taking into account that $X,Y\in\{e_1,e_3,e_4\}$ and using (\ref{eq::ABCD::Gauss1})--(\ref{eq::ABCD::Gauss3}) we obtain that $E_{k+1}^i(X,Y)$
can be expressed as a linear combination of $E_{k}^{i-1}(Z,W)$, where $Z,W\in\{e_1,e_3,e_4\}$. Since $i>3$ we have $i-1\geq 3$ and by assumption $E_{k}^{i-1}(Z,W)=0$. Now it follows that $E_{k+1}^i(X,Y)=0$.
\par Assume now that $i=3$, then we have
\begin{align*}
E_{k+1}^3(X,Y) &= -R^k\omega(e_3,e_4,R(e_3,e_4)X,Y,e_3,e_4,\ldots,e_3,e_4)\\
               &\phantom{=}-R^k\omega(e_3,e_4,X,R(e_3,e_4)Y,e_3,e_4,\ldots,e_3,e_4).
\end{align*}
For any pair $(X,Y)$ there exists $i,j\in\{1,3,4\}$ such that $X=e_i$ and $Y=e_j$
Since $E_{k+1}^3(X,Y)$ is antisymmetric relative to $X,Y$ it is enough to show that $E_{k+1}^3(e_i,e_j)=0$ for $i,j\in\{1,3,4\}$, $i<j$.
We have the following possibilities:
\begin{enumerate}
\item[(i)] $(X,Y)=(e_1,e_3)$. In this case we have
           \begin{align*}
           E_{k+1}^3(X,Y)&=-R^k\omega(e_3,e_4,e_1,R(e_3,e_4)e_3,e_3,e_4,\ldots,e_3,e_4)\\
                         &=-R^k\omega(e_3,e_4,e_1,-\gamma e_3\pm\gamma e_4,e_3,e_4,\ldots,e_3,e_4)\\
                         &=\gamma D_k\mp\gamma B_k =0,
           \end{align*}
           where the last equality follows from Lemma \ref{lm::ABCD}.\\
\item[(ii)] $(X,Y)=(e_1,e_4)$ In this case we have
           \begin{align*}
           E_{k+1}^3(X,Y)&=-R^k\omega(e_3,e_4,e_1,R(e_3,e_4)e_4,e_3,e_4,\ldots,e_3,e_4)\\
                         &=-R^k\omega(e_3,e_4,e_1,\mp\gamma e_3+\gamma e_4,e_3,e_4,\ldots,e_3,e_4)\\
                         &=\pm\gamma D_k-\gamma B_k =0,
           \end{align*}
           where the last equality is also consequence of Lemma \ref{lm::ABCD}.\\
\item[(iii)] $(X,Y)=(e_3,e_4)$ In this case $E_{k+1}^3(X,Y)=0$ thanks to (\ref{eq::ABCD::Gauss2}) and (\ref{eq::ABCD::Gauss3}).
\end{enumerate}
\par Summarising we have shown that $E_{k+1}^i(X,Y)=0$ for all $i\in\{3,\ldots,k+2\}$. Now by induction principle $E_{k}^i(X,Y)=0$ for all $k\geq 2$
$i\in\{3,\ldots,k+1\}$ and $X,Y\in\{e_1,e_3,e_4\}$.
\end{proof}
As a consequence of Lemma \ref{lm::Eki} one may prove the following
\begin{lem}\label{lm::oEkiSX}
If $S$ and $h$  are of the form {\upshape(\ref{eq::Sh22})} and $\alpha=\pm\gamma$ and $\beta=\mp\gamma$ then for every $k\geq 1$, $i\in\{1,\ldots,k+1\}$
and $X,Y\in\{e_3,e_4\}$ we have $E_k^i(SX,Y)=0$.
\end{lem}
\begin{proof}
For $X\in\{e_3,e_4\}$ we have that $SX$ is a linear combination of $e_3$ and $e_4$. Since $E_k^i$ is antisymmetric $2$-form we conclude that there exists a constant $c_0\in \mathbb{R}$ such that
$$
E_k^i(SX,Y)=c_0\cdot E_k^i(e_3,e_4).
$$
Now, if $k\geq 2$ the thesis follows from Lemma \ref{lm::Eki}. If $k=1$ we check by direct computation that $E_1^i(e_3,e_4)=0$ for $i=1,2$.
\end{proof}

Now we are at the position to prove the following lemma:
\begin{lem}\label{lm::AkCk}
If $S$ and $h$ are of the form {\upshape(\ref{eq::Sh22})} and $\alpha=\pm\gamma$ and $\beta=\mp\gamma$ then for every $k\geq 1$
we have
\begin{align}
\label{eq::Ak} A_{k+1}=-\lambda_1(C_k+D_k),\\
\label{eq::Ck} C_{k+1}=-\lambda_1(A_k+B_k).
\end{align}
\end{lem}
\begin{proof}
We compute
\begin{align*}
A_{k+1}&=R^{k+1}\omega(e_1,e_4,e_3,e_4,\ldots,e_3,e_4)\\
       &=R(e_1,e_4)\cdot R^k\omega(e_3,e_4,\ldots,e_3,e_4)\\
       &=-R^k\omega(R(e_1,e_4)e_3,e_4,\ldots,e_3,e_4)\\
       &\phantom{=}-R^k\omega(e_3,R(e_1,e_4)e_4,\ldots,e_3,e_4)\\
       &\cdots\\
       &\phantom{=}-R^k\omega(e_3,e_4,\ldots,e_3,R(e_1,e_4)e_4).\\
\end{align*}
Since $R(e_1,e_4)e_3=0$ and $R(e_1,e_4)e_4=-\lambda_1$ the above formula can be simplified as follows:
\begin{align}
\nonumber A_{k+1}&=\lambda_1 R^k\omega(e_3,e_1,e_3,e_4,\ldots,e_3,e_4)\\
\nonumber    &\phantom{=}+\lambda_1 R^k\omega(e_3,e_4,e_3,e_1,\ldots,e_3,e_4)\\
\nonumber    &\cdots\\
\nonumber    &\phantom{=}+\lambda_1 R^k\omega(e_3,e_4,e_3,e_4,\ldots,e_3,e_1)\\
\nonumber    &=\lambda_1\cdot \sum_{i=1}^{k+1}E_k^{i}(e_3,e_1).
\end{align}
If $k=1$ we have
$$
A_2=\lambda_1(E_1^{1}(e_3,e_1)+E_1^{2}(e_3,e_1))=-\lambda_1(C_1+D_1).
$$
If $k\geq 2$, by Lemma \ref{lm::Eki},  $E_k^{i}(e_3,e_1)=0$ for $i=3,\ldots ,k+1$. That is we obtain
$$
A_{k+1}=\lambda_1(E_k^{1}(e_3,e_1)+E_k^{2}(e_3,e_1))=-\lambda_1(C_k+D_k).
$$
Eventually we have shown (\ref{eq::Ak}).
The proof of (\ref{eq::Ck}) is similar.
\end{proof}


\begin{lem}\label{lm::BetaGammaEqual}
If $S$ and $h$ are of the form {\upshape(\ref{eq::Sh22})} and $\alpha=\pm\gamma$ and $\beta=\mp\gamma$ then for every $k\geq 0$ we have
\begin{eqnarray}\label{eq::AlphaBetaEqGamma:2}
A_{2k+1}=-\lambda_1^{2k+1}\omega(e_{1},e_{3}),
\end{eqnarray}
\begin{eqnarray}\label{eq::AlphaBetaEqGamma:1}
C_{2k+1}=-\lambda_1^{2k+1}\omega(e_{1},e_{4}).
\end{eqnarray}

\end{lem}
\begin{proof}
By straightforward computations we get
\begin{align*}
A_1&=R\omega(e_{1},e_{4},e_{3},e_{4})=-\lambda_1\omega(e_{1},e_{3}),\\
B_1&=\pm\gamma(\omega(e_{1},e_{3})\mp\omega(e_{1},e_{4})),\\
C_1&=R\omega(e_{1},e_{3},e_{3},e_{4})=-\lambda_1\omega(e_{1},e_{4}),\\
D_1&=\gamma(\omega(e_{1},e_{3})\mp\omega(e_{1},e_{4})).
\end{align*}
By Lemma \ref{lm::ABCD} we also have
\begin{align*}
B_{k+1}&=\pm\gamma(C_k\mp A_k),\\
D_{k+1}&=\gamma(C_k\mp A_k)
\end{align*}
for all $k\geq 1$. Summarising we have
\begin{align}
\label{eq::DkBk}D_k=\pm B_k
\end{align}
for $k\geq 1$.
Now, using Lemma \ref{lm::AkCk} we obtain
\begin{align*}
C_{k+1}\mp A_{k+1} &= -\lambda_1(A_k+B_k)\pm\lambda_1(C_k+D_k)\\
                   &= \pm\lambda_1(C_k\mp A_k+D_k\mp B_k)\\
                   &= \pm\lambda_1(C_k\mp A_k)
\end{align*}
where the last equality is a consequence of (\ref{eq::DkBk}).
The above implies, that $C_k\mp A_k$ is a geometric sequence, that is for $k\geq 1$ we have
$$
C_k\mp A_k = (\pm\lambda_1)^{k-1}(C_1\mp A_1).
$$
In particular we obtain explicit formulas for $B_{k+1}$ and $D_{k+1}$:
\begin{align}
\label{eq::BkExplicit}B_{k+1} &= \pm\gamma(\pm\lambda_1)^{k-1}(C_1\mp A_1),\\
\label{eq::DkExplicit}D_{k+1} &= \gamma(\pm\lambda_1)^{k-1}(C_1\mp A_1).
\end{align}
Using (\ref{eq::Ak})--(\ref{eq::Ck}) and (\ref{eq::DkExplicit}) for all $k\geq1$ we have
\begin{align*}
A_{2k+1} &= -\lambda_1(C_{2k}+D_{2k})\\
         &= \lambda_1^2(A_{2k-1}+B_{2k-1})-\lambda_1 D_{2k}\\
         &= \lambda_1^2 A_{2k-1}+\lambda_1^2 B_{2k-1}-\lambda_1\gamma(\pm\lambda_1)^{2k-2}(C_1\mp A_1)\\
         &= \lambda_1^2 A_{2k-1}+\lambda_1^2 B_{2k-1} -\gamma\lambda_1^{2k-1}(C_1\mp A_1).
\end{align*}
If $k=1$ we directly check that
$$
\lambda_1^2 B_{1} -\gamma\lambda_1(C_1\mp A_1)=0.
$$
If $k>1$, using (\ref{eq::BkExplicit}) we obtain
\begin{align*}
\lambda_1^2 B_{2k-1} -\gamma\lambda_1^{2k-1}(C_1\mp A_1)
         &= \lambda_1^2 (\pm\gamma(\pm\lambda_1)^{2k-3}(C_1\mp A_1)) -\gamma\lambda_1^{2k-1}(C_1\mp A_1)\\
         &= \lambda_1^2\gamma\lambda_1^{2k-3}(C_1\mp A_1)-\gamma\lambda_1^{2k-1}(C_1\mp A_1)\\
         &=0.
\end{align*}
Finally, for any $k\geq 1$ we have
$$
A_{2k+1}=\lambda_1^2 A_{2k-1}.
$$
Since $A_1=-\lambda_1\omega(e_{1},e_{3})$ we immediately get (\ref{eq::AlphaBetaEqGamma:2}).\\
In a similar way one may show that
$$
C_{2k+1}=\lambda_1^2 C_{2k-1}.
$$
and in consequence (\ref{eq::AlphaBetaEqGamma:1}).
\end{proof}

To simplify further computations we need to introduce the following notation:
\begin{align*}
T^k_{p,q,r}(X,Y)&:=R^{k}\omega(\underbrace{e_3,e_4,\ldots,e_3,e_4}_{2p},e_1,X,\underbrace{e_3,e_4,\ldots,e_3,e_4}_{2q},e_1,Y,\underbrace{e_3,e_4,\ldots,e_3,e_4}_{2r}),\\
U^k_{p,q,r}(X,Y)&:=-\lambda_1 h(X,e_3)T^k_{p,q,r}(e_4,Y)+\lambda_1 h(X,e_4)T^k_{p,q,r}(e_3,Y),\\
\hat{U}^k_{p,q,r}(X,Y)&:=-\lambda_1 h(X,e_3)T^k_{p,q,r}(Y,e_4)+\lambda_1 h(X,e_4)T^k_{p,q,r}(Y,e_3).
\end{align*}
where $k\geq 1$, $p,q,r\geq 0$, $p+q+r=k-1$.

We have the following lemma:

\begin{lem}\label{lm::Tk0qr}
If $S$ and $h$ are of the form {\upshape(\ref{eq::Sh22})} and $\alpha=\pm\gamma$ and $\beta=\mp\gamma$ then for every $k\geq 1$,
$q,r\geq 0$, $q+r=k$ and $X,Y\in\{e_3,e_4\}$ we have
\begin{align}\label{eq::TwithUUhat}
T^{k+1}_{0,q,r}(X,Y)=\sum_{i=0}^{q-1}U^k_{i,q-1-i,r}(X,Y)+\sum_{i=0}^{r-1}\hat U^k_{q,i,r-1-i}(X,Y).
\end{align}
Note that it may happen that $q=0$ (respectively $r=0$) in such case the sum $\sum_{i=0}^{q-1}$ (respectively the sum $\sum_{i=0}^{r-1}$)
is not present in the above formula.
\end{lem}
\begin{proof}
We compute
\begin{align*}
T^{k+1}_{0,q,r}(X,Y)&=R(e_1,X)\cdot R^k\omega(\underbrace{e_3,e_4,\ldots,e_3,e_4}_{2q},e_1,Y,\underbrace{e_3,e_4,\ldots,e_3,e_4}_{2r})\\
&=-R^k\omega(\underbrace{R(e_1,X)e_3,e_4,\ldots,e_3,e_4}_{2q},e_1,Y,\underbrace{e_3,e_4,\ldots,e_3,e_4}_{2r})\\
&\phantom{=}\cdots \\
&\phantom{=}-R^k\omega(\underbrace{e_3,e_4,\ldots,e_3,R(e_1,X)e_4}_{2q},e_1,Y,\underbrace{e_3,e_4,\ldots,e_3,e_4}_{2r})\\
&\phantom{=}-R^k\omega(\underbrace{e_3,e_4,\ldots,e_3,e_4}_{2q},R(e_1,X)e_1,Y,\underbrace{e_3,e_4,\ldots,e_3,e_4}_{2r})\\
&\phantom{=}-R^k\omega(\underbrace{e_3,e_4,\ldots,e_3,e_4}_{2q},e_1,R(e_1,X)Y,\underbrace{e_3,e_4,\ldots,e_3,e_4}_{2r})\\
&\phantom{=}-R^k\omega(\underbrace{e_3,e_4,\ldots,e_3,e_4}_{2q},e_1,Y,\underbrace{R(e_1,X)e_3,e_4,\ldots,e_3,e_4}_{2r})\\
&\phantom{=}\cdots \\
&\phantom{=}-R^k\omega(\underbrace{e_3,e_4,\ldots,e_3,e_4}_{2q},e_1,Y,\underbrace{e_3,e_4,\ldots,e_3,R(e_1,X)e_4}_{2r}).
\end{align*}
Using the Gauss equation we obtain
\begin{align*}
T^{k+1}_{0,q,r}(X,Y)&=\sum_{i=0}^{q-1}U^k_{i,q-1-i,r}(X,Y)\\
&\phantom{=}-R^k\omega(\underbrace{e_3,e_4,\ldots,e_3,e_4}_{2q},-SX,Y,\underbrace{e_3,e_4,\ldots,e_3,e_4}_{2r})\\
&\phantom{=}-R^k\omega(\underbrace{e_3,e_4,\ldots,e_3,e_4}_{2q},e_1,h(X,Y)\lambda_1 e_1,\underbrace{e_3,e_4,\ldots,e_3,e_4}_{2r})\\
&\phantom{=}+\sum_{i=0}^{r-1}\hat U^k_{q,i,r-1-i}(X,Y)\\
&=\sum_{i=0}^{q-1}U^k_{i,q-1-i,r}(X,Y)+E_k^{q+1}(SX,Y)+\sum_{i=0}^{r-1}\hat U^k_{q,i,r-1-i}(X,Y).
\end{align*}
Since $X,Y\in\{e_3,e_4\}$ Lemma \ref{lm::oEkiSX} implies (\ref{eq::TwithUUhat}).
\end{proof}

Now we can prove

\begin{lem}\label{lm::TkpqrXYproperties}
If $S$ and $h$ are of the form {\upshape(\ref{eq::Sh22})} and $\alpha=\pm\gamma$ and $\beta=\mp\gamma$ then for every $k\geq 1$ we have
\begin{align}\label{eq::Rke1e3etc}
T^k_{p,q,r}(e_3,e_3)=T^k_{p,q,r}(e_4,e_4)=\pm T^k_{p,q,r}(e_3,e_4)=\pm T^k_{p,q,r}(e_4,e_3)
\end{align}
where $p,q,r\geq 0$ and $p+q+r=k-1$
\end{lem}
\begin{proof}
For $k=1$, by straightforward computations we check that
\begin{align*}
R\omega(e_1,e_3,e_1,e_3)=R\omega(e_1,e_4,e_1,e_4)=\pm R\omega(e_1,e_3,e_1,e_4)=\pm R\omega(e_1,e_4,e_1,e_3)
\end{align*}
so
$$
T^1_{0,0,0}(e_3,e_3)=T^1_{0,0,0}(e_4,e_4)=\pm T^1_{0,0,0}(e_3,e_4)=\pm T^1_{0,0,0}(e_4,e_3).
$$
Assume now that (\ref{eq::Rke1e3etc}) is true for some $k\geq 1$ and all $p,q,r\geq0$ such that $p+q+r=k-1$.
We compute
\begin{align}\label{eq::Uk}
U^k_{p,q,r}(e_3,e_3)=-\lambda_1 T^k_{p,q,r}(e_4,e_3) = -\lambda_1 T^k_{p,q,r}(e_3,e_4)\\
\nonumber=U^k_{p,q,r}(e_4,e_4)=\pm U^k_{p,q,r}(e_3,e_4)=\pm U^k_{p,q,r}(e_4,e_3).
\end{align}
In a similar way w get
\begin{align}\label{eq::Uhatk}
\hat U^k_{p,q,r}(e_3,e_3)= \hat U^k_{p,q,r}(e_4,e_4) =\pm \hat U^k_{p,q,r}(e_3,e_4)=\pm \hat U^k_{p,q,r}(e_4,e_3).
\end{align}
Now, let us consider $T^{k+1}_{p,q,r}(X,Y)$ where $p,q,r\geq 0$ and $p+q+r=k$.
\par If $p=0$ then from Lemma \ref{lm::Tk0qr} we have
\begin{align*}
T^{k+1}_{0,q,r}(X,Y)=\sum_{i=0}^{q-1}U^k_{i,q-1-i,r}(X,Y)+\sum_{i=0}^{r-1}\hat U^k_{q,i,r-1-i}(X,Y).
\end{align*}
Using (\ref{eq::Uk})--(\ref{eq::Uhatk}) we obtain
\begin{align}\label{eq:Tk0-properties}
T^{k+1}_{0,q,r}(e_3,e_3)=T^{k+1}_{0,q,r}(e_4,e_4)=\pm T^{k+1}_{0,q,r}(e_3,e_4)=\pm T^{k+1}_{0,q,r}(e_4,e_3).
\end{align}
\par Assume now that $p>0$.
First note that
\begin{align}\label{eq::Rke3e4property}
R^k\omega(\ldots,R(e_3,e_4)e_3,e_4,\ldots) &= R^k\omega(\ldots,-\gamma e_3\pm\gamma e_4,e_4,\ldots)\\
\nonumber&=-\gamma R^k\omega(\ldots,e_3,e_4,\ldots) \\
\nonumber&= R^k\omega(\ldots,e_3,\pm\gamma e_3-\gamma e_4,\ldots)\\
\nonumber&= -R^k\omega(\ldots,e_3,R(e_3,e_4)e_4,\ldots).
\end{align}
By (\ref{eq::Rke3e4property}) and using the fact that $R(e_3,e_4)e_1=0$ we get
\begin{align*}
&T^{k+1}_{p,q,r}(X,Y)\\
&=R(e_3,e_4)\cdot R^k\omega(\underbrace{e_3,e_4,\ldots,e_3,e_4}_{2p-2},e_1,X,\underbrace{e_3,e_4,\ldots,e_3,e_4}_{2q},e_1,Y,\underbrace{e_3,e_4,\ldots,e_3,e_4}_{2r})\\
&=-R^k\omega(\underbrace{e_3,e_4,\ldots,e_3,e_4}_{2p-2},e_1,R(e_3,e_4)X,\underbrace{e_3,e_4,\ldots,e_3,e_4}_{2q},e_1,Y,\underbrace{e_3,e_4,\ldots,e_3,e_4}_{2r})\\
&\phantom{=}-R^k\omega(\underbrace{e_3,e_4,\ldots,e_3,e_4}_{2p-2},e_1,X,\underbrace{e_3,e_4,\ldots,e_3,e_4}_{2q},e_1,R(e_3,e_4)Y,\underbrace{e_3,e_4,\ldots,e_3,e_4}_{2r})\\
&=-T^{k}_{p-1,q,r}(R(e_3,e_4)X,Y)-T^{k}_{p-1,q,r}(X,R(e_3,e_4)Y).
\end{align*}
Using (\ref{eq::ABCD::Gauss2}), (\ref{eq::ABCD::Gauss3}) and (\ref{eq::Rke1e3etc}), by direct computation, one may check that
$$
T^{k}_{p-1,q,r}(R(e_3,e_4)X,Y)=T^{k}_{p-1,q,r}(X,R(e_3,e_4)Y)=0
$$
for any $X,Y\in\{e_3,e_4\}$. In consequence we get
\begin{align}\label{eq::Tk+1=0}
T^{k+1}_{p,q,r}(X,Y) = 0
\end{align}
for all $X,Y\in\{e_3,e_4\}$. Now from (\ref{eq::Tk+1=0}) and (\ref{eq:Tk0-properties}) we obtain
$$
T^{k+1}_{p,q,r}(e_3,e_3)=T^{k+1}_{p,q,r}(e_4,e_4)=\pm T^{k+1}_{p,q,r}(e_3,e_4)=\pm T^{k+1}_{p,q,r}(e_4,e_3)
$$
for every $p,q,r\geq 0$, $p+q+r=k$.
By induction principle the formula (\ref{eq::Rke1e3etc}) is true for any $k\geq 1$.
\end{proof}

As an immediate consequence of Lemma \ref{lm::TkpqrXYproperties} (see formula (\ref{eq::Tk+1=0})) we get the following


\begin{cor}\label{wn::TkpqrXYwhen0}
If $S$ and $h$ are of the form {\upshape(\ref{eq::Sh22})} and $\alpha=\pm\gamma$ and $\beta=\mp\gamma$ then for every $k\geq 2$, $p\geq 1$, $q,r\geq 0$
and $p+q+r=k-1$ we have
\begin{align}
T^k_{p,q,r}(X,Y)=0
\end{align}
for any $X,Y\in\{e_3,e_4\}$.
\end{cor}

The above lemmas and corollary alow us to prove the following

\begin{lem}\label{lm::Tk00k-1}
If $S$ and $h$ are of the form {\upshape(\ref{eq::Sh22})} and $\alpha=\pm\gamma$ and $\beta=\mp\gamma$ then for every $k\geq 2$
we have
\begin{align}\label{eq::formulaTk00k-1}
T^k_{0,0,k-1}(e_3,e_3)=2^{k-2}(\mp\lambda_1)^{k-1}\gamma\omega(e_3,e_4).
\end{align}
\end{lem}
\begin{proof}
\par Let us cosider $T^{k+1}_{0,q,r}(e_3,e_3)$, when $q+r=k$, $k\geq 1$.
If $q\geq 1$, using Lemma \ref{lm::Tk0qr} and Corollary \ref{wn::TkpqrXYwhen0} (if $k\geq 2$) we obtain
\begin{align*}
T^{k+1}_{0,q,r}(e_3,e_3) &= \sum_{i=0}^{q-1}U^k_{i,q-1-i,r}(e_3,e_3)+\sum_{i=0}^{r-1}\hat U^k_{q,i,r-1-i}(e_3,e_3)\\
                         &= -\lambda_1\sum_{i=0}^{q-1}T^k_{i,q-1-i,r}(e_4,e_3)-\lambda_1\sum_{i=0}^{r-1}T^k_{q,i,r-1-i}(e_3,e_4)\\
                         &= -\lambda_1 T^{k}_{0,q-1,r}(e_4,e_3).
\end{align*}
Now, by Lemma \ref{lm::TkpqrXYproperties} we have
\begin{align*}
T^{k+1}_{0,q,r}(e_3,e_3) = -\lambda_1 T^{k}_{0,q-1,r}(e_4,e_3)= \mp\lambda_1 T^{k}_{0,q-1,r}(e_3,e_3)
\end{align*}
that is
\begin{align}\label{eq::Tk+1qr}
T^{k+1}_{0,q,r}(e_3,e_3) = (\mp\lambda_1)^q T^{r+1}_{0,0,r}(e_3,e_3).
\end{align}
If $q=0$, by Lemma \ref{lm::Tk0qr} and Lemma \ref{lm::TkpqrXYproperties} we have
\begin{align}\label{eq::Tk+1sum}
T^{k+1}_{0,0,k}(e_3,e_3) = -\lambda_1 \sum_{i=0}^{k-1}T^{k}_{0,i,k-1-i}(e_3,e_4)=  \mp\lambda_1 \sum_{i=0}^{k-1}T^{k}_{0,i,k-1-i}(e_3,e_3)
\end{align}
for $k\geq 1$.
Applying (\ref{eq::Tk+1qr}) to (\ref{eq::Tk+1sum}) we get
\begin{align}\label{eq::Tk+1sum2}
T^{k+1}_{0,0,k}(e_3,e_3) = \mp\lambda_1 \sum_{i=0}^{k-1}(\mp\lambda_1)^iT^{k-i}_{0,0,k-i-1}(e_3,e_3).
\end{align}
By straightforward computations we check that
$$
T^1_{0,0,0}(e_3,e_3)=\gamma\omega(e_3,e_4)\quad\text{and}\quad T^2_{0,0,1}(e_3,e_3)=\mp\lambda_1\gamma\omega(e_3,e_4)
$$
so in particular (\ref{eq::formulaTk00k-1}) is true for $k=2$.
Let us fix $k_0\geq 2$ and assume that (\ref{eq::formulaTk00k-1}) is true for any $k\in\{2,\ldots,k_0\}$. Now we have
\begin{align*}\label{eq::Tk+1sum3}
T^{k_0+1}_{0,0,k_0}(e_3,e_3) &= \mp\lambda_1 \sum_{i=0}^{k_0-2}(\mp\lambda_1)^i  2^{k_0-i-2}(\mp\lambda_1)^{k_0-i-1}\gamma\omega(e_3,e_4)\\
&\phantom{=}+(\mp\lambda_1)^{k_0}T^1_{0,0,0}(e_3,e_3)\\
&=(\mp\lambda_1)^{k_0}\gamma\omega(e_3,e_4)\sum_{i=0}^{k_0-2}2^{k_0-i-2}+(\mp\lambda_1)^{k_0}\gamma\omega(e_3,e_4)\\
&=2^{k_0-1}(\mp\lambda_1)^{k_0}\gamma\omega(e_3,e_4).
\end{align*}
That is (\ref{eq::formulaTk00k-1}) holds also for $k=k_0+1$.
Now, by induction principle (\ref{eq::formulaTk00k-1}) is true for any $k\geq 2$.
\end{proof}

Now we are ready to prove main results of this paper.
Namely we have


\begin{thrm}\label{tw::RKomega0Lorentzian4dim}
Let $f\colon \M\rightarrow\R^{5}$ be a non-degenerate affine hypersurface with a locally equiaffine transversal vector field  $\xi$
and an almost symplectic form $\omega$. If $R^k\omega=0$ for some $k\geq 1$ and the second fundamental form is Lorentzian on $M$ (that is has signature $(3,1)$)
then the shape operator $S$ has the rank $\leq 1$.
\end{thrm}
\begin{proof}
Let $x\in M$ and let $\{e_1,\ldots,e_{4}\}$ be the basis from Lemma \ref{lm::PrzypNiemozliwy}. If $S$ and $h$ are of the form (\ref{eq::Sh21}),
then in the same way as in the proof of Theorem \ref{tw::RKOmega0} (see \cite{MSz} for details) we obtain that  $S$ is equal to zero thus $\rank S_x=0$.
\par Let $S$ and $h$ have the form (\ref{eq::Sh22}). If $\beta^2-\gamma^2\neq 0$ then by Lemma \ref{lm::BetaGamma} we can change
the basis $\{e_1,\ldots,e_{4}\}$ of $T_xM$ to $h$-ortonormal basis ${e_1',\ldots,e_{4}'}$ in such a way that $S$ and $h$ are of the form (\ref{eq::Sh22new}). In particular, $S$ is diagonal. Since $\beta^2-\gamma^2\neq 0$, Corollary \ref{cor::detZero} implies that $\alpha +\beta\neq 0$, thus $\rank S_x\geq 1$. However, since $S$ is diagonal we again can use methods from \cite{MSz} (proof of Theorem \ref{tw::RKOmega0}) and show that $\rank S_x=0$, what leads to contradiction. It means that the case $\beta^2-\gamma^2\neq 0$ is not possible.
\par Assume now that $\beta^2-\gamma^2 = 0$. By Corollary \ref{cor::detZero} we  have
$\alpha\beta+\gamma^2=0$ and in consequence we get that $\alpha=\pm\gamma$ and $\beta=\mp\gamma$.
Without loss of generality (rearranging $e_1$ and $e_2$ if needed) we may always assume that $|\lambda_1|\geq |\lambda_2|\geq 0$.
If $\lambda_1=\lambda_2=0$ then $\rank S_x=1$ (since $\gamma\neq 0$) and the proof is completed. Let as assume that  $\lambda_1\neq 0$.
Since $R^k\omega=0$ for some $k\geq 1$ then in particular  $R^{2k}\omega=0$ and $R^{2k+1}\omega=0$. Now by Lemma \ref{lm::BetaGammaEqual} we immediately obtain
$\omega(e_{1},e_{3})=\omega(e_{1},e_{4})=0$. Since $\omega$ is non-degenerate then
\begin{align*}
\det\omega&=(\omega(e_1,e_2)\omega(e_3,e_4)-\omega(e_1,e_3)\omega(e_2,e_4)+\omega(e_1,e_4)\omega(e_2,e_3))^2\\
&=(\omega(e_1,e_2)\omega(e_3,e_4))^2\neq 0.
\end{align*}
In particular $\omega(e_{3},e_{4})\neq 0$.  Now Lemma
\ref{lm::Tk00k-1} implies that $\lambda_1\gamma = 0$ what (since $\gamma\neq 0$ and $\lambda_1\neq 0$) leads us to contradiction. Summarising
we must have $\lambda_1=0$ and in consequence also $\lambda_2=0$.
\end{proof}

From Theorem \ref{tw::RKomega0Lorentzian4dim} we directly obtain the following

\begin{thrm}\label{tw::NablaKOmega0Lorentzian4dim}
Let $f\colon \M\rightarrow\R^{5}$ be a non-degenerate affine hypersurface with a locally equiaffine transversal vector field  $\xi$
and an almost symplectic form $\omega$. If $\nabla^k\omega=0$ for some $k\geq 1$ and the second fundamental form is Lorentzian on $M$ (that is has signature $(3,1)$)
then the shape operator $S$ has the rank $\leq 1$.
\end{thrm}
\begin{proof}
If $\nabla^k\omega = 0$ for some $k$ then, of course, we have that also $\nabla^{2k}\omega = 0$ and now by Lemma \ref{lm::WzorNaRkT} we get $R^{k}\omega=0$.
Now, thesis is an immediate consequence of Theorem \ref{tw::RKomega0Lorentzian4dim}.
\end{proof}

\emph{This Research was financed by the Ministry of Science and Higher Education of the Republic of Poland.}

\end{document}